\documentclass[9pt]{amsart}
\usepackage{amsmath,amsfonts,amsthm,mathrsfs}
\usepackage{color}
\usepackage{verbatim}
\usepackage{eepic,epic}
\usepackage{epsfig,subfigure,epstopdf}
\usepackage{enumerate}
\usepackage[]{graphics}
\usepackage{multirow}
\usepackage[colorlinks,linktocpage,linkcolor=blue]{hyperref}
\usepackage{graphicx}
\usepackage{caption}
\usepackage{pdfpages}
\usepackage{chngcntr}
\counterwithin{table}{section}

\def\eat#1{}
\renewcommand{\paragraph}[1]{\par {\it #1}} 

\newtheorem{theorem}{Theorem}[section]
\newtheorem{proposition}{Proposition}[section]
\newtheorem{lemma}{Lemma}[section]

\newtheorem{example}[theorem]{Example}

\newtheorem{remark}{Remark}[section]
\newtheorem{corollary}{Corollary}[section]
\numberwithin{equation}{section}

\def\RR{{\mathbb R}}
\def\II{{(D)}}

\def\dH{\dot{H}}
\def\cC{{\mathcal{C}}}

\chardef\atsign='100

\def\cE{\mathcal E}
\def\cL{\mathcal L}

\def\cut#1\endcut{}

\def\hdiv(#1){H_{\hbox{div}}(#1)}

\def\cL{\mathcal{L}}

\def\DD{\mathcal{D}}

\def\RR{\mathbb {R}}
\def\CC{\mathbb {C}}

\def\Forall{\hbox{\ \  for all }}

\def\beq#1\eeq{\begin{equation}#1\end{equation}}
\def\bal#1\eal{\begin{aligned}#1\end{aligned}}
\def\baln#1\ealn{\begin{align*}#1\end{align*}}

\renewcommand{\Re}{\mathfrak{Re}}
\renewcommand{\Im}{\mathfrak{Im}}
\newcount\icount

\def\DD#1#2{\icount=#1
  \ifnum\icount<1
  \,_{ 0}\kern -.1em D^{#2}_{\kern -.1em x}
  \else
  \,_{x}\kern -.2em D^{#2}_1
  \fi
}

\def\DDRI#1#2{\icount=#1
  \ifnum\icount<1
  \,_{-\infty}^{\kern 1em R}\kern -.2em D^{#2}_{\kern -.1em x}
  \else
  \,_{x}^R \kern -.2em D^{#2}_\infty
  \fi
}

\def\DDR#1#2{\icount=#1
  \ifnum\icount<1
 _{0}^{ \kern -.1em R} \kern -.2em D^{#2}_{\kern -.1em x}
  \else
Translate
EnglishSpanishFrenchDetect languageChinese (Simplified)EnglishSpanish
 _{x}^{ \kern -.1em R} \kern -.2em D^{#2}_{\kern -.1em 1}
  \fi
}

\def\DDCI#1#2{\icount=#1
  \ifnum\icount<1
  \,_{-\infty}^{\kern 1em C}  \kern -.2em D^{#2}_{\kern -.1em x}
  \else
  \,_{x}^C \kern -.2em  D^{#2}_\infty
  \fi
}

\def\DDC#1#2{\icount=#1
  \ifnum\icount<1
  \,_{0}^C \kern -.2em  D^{#2}_{\kern -.1em x}
  \else
  \,_{x}^C \kern -.2em D^{#2}_1
  \fi
}

\def\Hdi#1#2{\icount=#1
  \ifnum\icount<1
  \widetilde H_{L}^{#2}\II
  \else
  \widetilde H_{R}^{#2}\II
  \fi
}

\begin{document}
\title[
PARABOLIC EQUATIONS INVOLVING FRACTIONAL POWERS]
{THE APPROXIMATION OF PARABOLIC EQUATIONS INVOLVING FRACTIONAL POWERS OF ELLIPTIC OPERATORS}

\author{Andrea Bonito}
\address{Department of Mathematics, Texas A\&M University, College Station,
TX~77843-3368.}
\email{bonito\atsign math.tamu.edu}

\author{Wenyu Lei}
\address{Department of Mathematics, Texas A\&M University, College Station,
TX~77843-3368.}
\email{wenyu\atsign math.tamu.edu}

\author{Joseph E. Pasciak}
\address{Department of Mathematics, Texas A\&M University, College Station,
TX~77843-3368.}
\email{pasciak\atsign math.tamu.edu}

\date{\today}
\begin{abstract}
We study the numerical approximation of a time dependent equation involving 
fractional powers of an elliptic operator $L$ defined to be the 
unbounded 
operator associated with a Hermitian, coercive and 
bounded sesquilinear form on $H^1_0(\Omega)$. 
The time dependent solution $u(x,t)$ is represented as  a Dunford Taylor 
integral along a contour in the complex plane.

The contour integrals are approximated using 
sinc quadratures.
In the case of homogeneous right-hand-sides and initial value $v$, the
approximation results in a linear combination of functions
$(z_qI-L)^{-1}v\in H^1_0(\Omega)$ for a
finite number of quadrature points $z_q$ lying along the contour.
In turn, these quantities are approximated using complex valued continuous piecewise linear finite elements.

Our main result provides $L^2(\Omega)$ error estimates between the solution $u(\cdot,t)$ and its final approximation.
Numerical results illustrating the behavior of the algorithms are provided.
\end{abstract}

\maketitle

\section{Introduction}
We consider the approximation of parabolic equations
where the elliptic part is given by a fractional power of an elliptic
boundary value operator.
  This is a prototype for 
time dependent equations with integral
or nonlocal
operators and has numerous applications \cite{bigQG,Constantin_Wu,cushman1993nonlocal,bakunin2008turbulence,bates2006some}.
For example, a  non-local operator results from replacing Brownian diffusion
with Levy-diffusion. 
\textcolor{black}{When the spatial domain is bounded, the fractional powers $L^\beta$  can
be defined in terms of  Fourier series, namely,
\beq
L^\beta v =\sum_{j=1}^\infty (v,\psi_j) \lambda_j^ \beta \psi_j.
\label{ldef}
\eeq
Here $(\cdot,\cdot)$ denotes the Hermitian $L^2(\Omega)$ inner product and 
$\{\psi_j\}$ is an $L^2(\Omega)$-orthonormal 
basis of eigenfunctions of $L$ with eigenvalues $\{\lambda_j\}$.
An alternate, yet equivalent, formulation can be found in \cite{kato1961,lunardi} for general regularly accretive operators $L$, see also \cite{BP2015}.}


In this paper, we focus on a bounded domain problem where $\RR^d$ is replaced by
a bounded domain $\Omega \subset \RR^d$ and for $T>0$, the targeted
function $u: \Omega \times [0,T] \rightarrow \mathbb R$ satisfies
\begin{equation}\label{eq:problem}
  \left\lbrace
  \begin{aligned}
      u_t+L^\beta u=f, & \qquad\text{in } \Omega\times (0,T),\\
      u=0, & \qquad\text{on }\partial   \Omega\times(0,T),\\
      u(t=0)=v, & \qquad\text{on }\Omega. \\
  \end{aligned}
\right.
\end{equation}
Here $v\in L^2(\Omega)$, $f \in L^2(0,T;L^2(\Omega))$, $\beta\in (0,1)$.
The following discussion focus on the case $f=0$ as the case $f \not = 0$ follows from it using the Duhamel principle (see e.g. Corollary~\ref{cor:nonho_semi}).
\textcolor{black}{
We point out that although we restrict our considerations to homogeneous Dirichlet boundary conditions,  other types of homogeneous boundary conditions can be treated similarly. We refer to \cite{frac14} for a general framework.
}

At the continuous level, \eqref{eq:problem} fits into the standard
theory for parabolic initial value problems.  The weak form $(L^\beta u,v)$ is a bounded coercive operator
on $D(L^{\beta/2})$ resulting in existence and uniqueness in the natural
spaces (see, Section~\ref{preli} for details). In contrast, the
situation is not standard at the discrete  (finite element) level.   
This is because stiffness matrix entries corresponding to the 
Galerkin method, $\{(L^\beta \phi_i,\phi_j)\}$, with $\{\phi_i\}$
denoting the finite element basis, cannot be evaluated exactly 
so that the classical analysis \cite{thomee} does not apply.  
Instead, we use a discrete approximation to the sesquilinear form
namely, $\{(L_h^\beta u,v)\}$ with $L_h$ being the finite element
approximation to $L$.

Several numerical methods to approximate the solution of
\eqref{eq:problem} with $f=0$ have been studied. 
One is based on 
the spectral decomposition of a (symmetric) finite difference approximation 
$L_h$ to $L$ 
(see \cite{ILIC1,ILIC2}). 
It follows from \eqref{ldef} that the solution to 
\eqref{eq:problem} is given by  
\begin{equation}\label{eq:spectral}
  u(x,t)=e^{-tL^\beta} v=\sum_{j=1}^\infty e^{-t\lambda_j^\beta}(v,\psi_j) \psi_j(x). 
\end{equation}
and \cite{ILIC1,ILIC2} propose the finite difference approximation
given by 
\beq
U_h=\textcolor{black}{e^{-tL_h^\beta}} V_h\label{ilic}
\eeq
with $V_h$ denoting the interpolant of $v$.  The right hand side of
\eqref{ilic} is computed from the spectral decomposition of $L_h$, where $L_h$ is the discrete Laplacian 
generated by the finite difference scheme. 
Of course, the direct implementation of
this method requires the computation of the discrete eigenvectors and
their eigenvalues. This is a demanding computational problem when the
dimension of the discrete problem becomes large.

A second  approach \cite{Nochetto_time} is to consider the fractional
power of $L$ as a ``Dirichlet to Neumann'' map 
via the ``so-called'' Caffarelli-Silvestre extension problem (see
\cite{Caffarelli_Silvestre, Stinga_Torrea}) 
on the semi-infinite cylinder $\Omega\times(0,\infty)$. The trace of
solution of the local extension problem
onto $\Omega$ is the solution of the  original nonlocal problem. Numerically, 
the extension problem can be approximated using finite element method 
in a bounded domain by truncating in the extra dimension to $(0,\mathcal Y)$
for some $\mathcal Y>1$. The truncation error in $\mathcal Y$ 
becomes exponentially
small as $\mathcal Y$ increases (see \cite{Nochetto_time}).  

The goal of this paper is twofold.  First, we study the convergence   
of \eqref{ilic} when the 
numerical approximation $L_h$ of $L$ is defined from the
Galerkin finite element method applied 
in a finite element  approximation space $H_h$.  
In this case, the eigenfunctions
$\{\psi_{j,h}\} $ of $L_h$ can be taken to be  $L^2(\Omega)$ orthonormal
functions in $H_h$ leading to the  approximation 
\begin{equation}\label{eq:spectral_finite}
  u_h(t)=\sum_{j=1}^M e^{-t\lambda_{j,h}^\beta}(v,\psi_{j,h}) \psi_{j,h}
\end{equation}
with $\lambda_{j,h}$ denoting the corresponding eigenvalues.

Motivated by \cite{frac14} on the steady state problem, we prove  (see, Theorem~\ref{thm:semi_init})
that  
 for some $\alpha\in(0,1]$ depending on $\Omega$ and $L$, there exists a constant $C(t)$ uniform in $h$, such that
    \begin{equation}\label{ineq:init_value}
   \|u(\cdot,t)-u_h(t)\| \le C(t)h^{2\alpha}.
    \end{equation}
The constant $C(t)$ depends on $t$ and $\alpha$, the regularity
of $v$. 

Based on the techniques presented in \cite{Hackbush_exponential,Hackbush_data_sparse,sinc_int,
Thomee_laplace,Thomee_parallel}, we next provide and study a
 sinc type quadrature method \cite{lund1992sinc}
 for approximating $u_h$ avoiding the eigenfunction expansion in \eqref{eq:spectral_finite}.   
We note that  $u_h(t)=e^{-tL_h^\beta} \pi_h v$ with
$\pi_h $ denoting the $L^2(\Omega)$ projection onto $H_h$.
Both $u$ and $u_h$ can be written using the Dunford-Taylor integral,
e.g.,  
\begin{equation}\label{eq:Dunford_int}
  u_h(t)=\frac{1}{2\pi i}\int_{\cC} e^{-tz^\beta}R_z(L_h)\pi_h v\ dz
\end{equation}
 (see, Section~\ref{preli})
where $R_z(L_h):=(zI-L_h)^{-1}$ is the resolvent.  
Following
\cite{Hackbush_data_sparse}, for $b\in (0,\lambda_1/\sqrt{2})$, 
we take $\cC$ given by the path
\beq
\gamma(y)=b(\cosh{ y}+i\sinh{y}), \qquad y\in \RR.
\label{hyper-curve}
\eeq
Then \eqref{eq:Dunford_int} becomes
\beq
u_h(t) = \frac{1}{2\pi i} \int_{-\infty}^\infty e^{-t \gamma(y)^\beta} \gamma'(y)[(\gamma(y)I-L_h)^{-1}\pi_h v] \, dy.
\label{realint-u_h}
\eeq
We apply the sinc method to approximate the vector valued integral in
\eqref{realint-u_h}.

We show that for fixed $t$, the  $L^2(\Omega)$ error 
between $u_h(t)$ and its numerical approximation with $2N+1$ quadrature
points is 
$O(e^{-N\log(N)})$ (Example~\eqref{Example_exp}).  Each quadrature point $y_\ell$ involves an evaluation of
$(\gamma(y_\ell)I-L_h)^{-1}\pi_h v$, i.e., the solution of a matrix problem
involving the stiffness matrix for the form 
$\gamma(y_\ell)( u,v)-A(u,v)$ and the usual
finite element right hand side vector for $v$.  Here $A(\cdot,\cdot)$
denotes the Hermitian form mentioned above (see also Section~\ref{preli}).
These problems
are independent and can be solved in parallel.
The total error is estimated by combining the space discretization error and the quadrature error (see Corrollary~\ref{c:total_error}).

\cut
Note that we anticipate a fully discrete scheme based on a backward Euler time discretization of \eqref{eq:problem} consist in finding recursively $u_h^{n+1} \in H_h$ such that
$$ 
u_h^{n+1}=\frac{1}{2\pi i}\int_{\cC} (1+\delta t z^\beta )^{-1}R_z(t^{1/\beta} L_h) (\delta t \pi_h f(t^{n+1})+u_h^{n})\ dz,
$$
where $\delta t$ is the time step discretization.
By letting the contour runs from $\infty e^{i\pi}$ to 0, and then from 0 to $\infty e^{-i\pi}$(see also \cite{kato1961}), the above expression becomes
$$
	u_h^{n+1}=\frac{\sin\pi\alpha}{\pi}\int_0^{\infty} 
	\frac{\mu^\beta}{1+2\mu^\beta\cos\pi\beta+\mu^{2\beta}}
	(\mu I+\delta t^{1/\beta}L_h)^{-1}(\delta t \pi_h f^{n+1}+u_h^n)\, d\mu
$$
and again a sinc quadrature scheme can also be applied to the above formula upon setting $\mu:=e^y$.
The results provided in this work will be instrumental to analyze the above full discretization.
\endcut

The outline of this paper is as follows. In Section \ref{preli}, we
provide some basic notations and preliminaries about fractional powers of unbounded operators. 
The finite element setting and the space
discretization scheme are developed in Section \ref{FEapprox}. The error
between $u(t)$ and $u_h(t)$ is also given there.
The quadrature scheme and its analysis are given in Section 
\ref{quad}.  Finally, some numerical results illustrating the convergence
behavior are given in Section \ref{numerical}.

\section{Preliminaries}\label{preli}

\paragraph{Notation.} Let $\Omega \subset \mathbb R^d$, $d = 2,3$, be a bounded polygonal domain with Lipschitz boundary.
We denote by $L^2(\Omega)$, $H^1(\Omega)$ and $H^1_0(\Omega)$ the standard Sobolev spaces of complex valued function with norms:
$$\bal
\| v\| &:= \|v\|_{L^2(\Omega)}:= \left( \int_\Omega |v|^2 \right)^{1/2},\\
 \|v\|_{H^1(\Omega)}&:= \left(\| v\|_{L^2(\Omega)}^2 + \| |\nabla v|
 \|_{L^2(\Omega)}^2\right)^{1/2} \ \text{and}\\
\|v\|_{H^1_0(\Omega)}&:= \| |\nabla v |\|_{L^2(\Omega)}.
\eal
$$

We use the notation
$$
A \leq c B \qquad \text{and} \qquad A \leq C B
$$
where $c$ and $C$ are generic constants independent of $A$ and $B$.

\paragraph{The Unbounded Operator $L$ and the Dotted Spaces.}
We recall that  $A(\cdot,\cdot)$ is assumed to be a Hermitian, coercive and 
bounded sesquilinear form on $H^1_0(\Omega)$.
This means that there are $c_0$ and $c_1$ two positive constants such that
$$\bal
A(v,v) &\geq c_0 \| v \|^2_{H^1_0(\Omega)},\Forall v\in
H^1_0(\Omega),\hbox{ and }\\
 |A(v,w)| &\leq c_1 \|
v\|_{H^1_0(\Omega)} \| w \|_{H^1_0(\Omega)},\Forall v,w\in H^1_0(\Omega).
\eal
$$
Let $T:L^2(\Omega)\rightarrow
H^1_0(\Omega)$ be the solution operator, i.e.,  $w:=Tf \in H^1_0(\Omega)$ is the unique solution (guaranteed by Lax-Milgram) of 
\beq
A(w,\theta)= (f,\theta),\Forall \theta\in H^1_0(\Omega),
\label{laxm}
\eeq
Following \cite{kato1961}, see also \cite{BP2015}, we define the
unbounded operator $L:=T^{-1}$ on $L^2(\Omega)$ with domain
$D(L):=\hbox{Range}(T)$.

We next define the dotted spaces for $s\ge 0$.
We note that since $T$ is compact and symmetric 
on $L^2(\Omega)$, Fredholm theory guarantees an $L^2(\Omega)$-orthogonal
basis of eigenfunctions $\{\psi_j\}_{j=1}^\infty$ with non-increasing 
real eigenvalues $\mu_1>\mu_2 \geq \mu_3 \geq ... >0$. 
For every positive integer $j$, $\psi_j$ is 
also an eigenfunction of $L$ with corresponding eigenvalue $\lambda_j=1/\mu_j$.
These eigenpairs are instrumental for defining the following dotted spaces.
For $s\geq 0$, the dotted space $\dH^s$ is given by 
$$
\dH^s:= \left\lbrace f\in L^2(\Omega)\ \text{s.t.} \  \sum_{j=1}^\infty \lambda_j^s |(f,\psi_j)|^2<\infty
\right\rbrace.
$$
These spaces form a Hilbert scale of interpolation spaces equipped with the norm
 $$
 \|v\|_{\dH^s} := \bigg(\sum_{j=1}^\infty \lambda_j^s |(v,\psi_j)|^2\bigg)^{1/2}.
 $$ 
Notice that 
$$\|v\|_{\dH^s}=(L^sv,v)^{1/2}=\|L^{s/2} v\|_{L^2(\Omega)},$$
i.e., $\dH^s = D(L^{s/2})$.
Moreover,  $\dH^1 $ coincides with $H^1_0(\Omega)$ and $\dH^0$
coincides with $L^2(\Omega)$, in both cases with equal norms \textcolor{black}{(see for e.g. Lemma 3.1 in \cite{thomee})}.

We denote  $\dH^{-s}$ to be the set of bounded anti-linear functionals
on 
$\dH^s$, i.e. if $\langle F,v\rangle$ denotes the action of $F \in \dH^{-s}$ applied to $v\in \dot{H}^s$ then
$$
\dH^{-s} = \left\{F:\dH^s \rightarrow \mathbb C \ \text{s.t} \ \|F\|_{\dH^{-s}}:= \bigg(\sum_{j=1}^\infty
\lambda_j^{-s}|\langle F,\psi_j\rangle |^{2}\bigg)^{1/2}<\infty\right \}.
$$
Since $H^1_0(\Omega)$ and $\dH^1$ coincide, so does $\dH^{-1}$ and $H^{-1}(\Omega)$,
the set of bounded anti-linear functionals on $H^1_0(\Omega)$.
\cut
By definition, the spaces are such that for $s\ge -1$, 
$T:\dH^s \rightarrow \dH^{s+2}$ 
and $L:\dH^{s+2}\rightarrow \dH^s $.
\endcut

\cut
THIS IS NOT TRUE!
It is well known that 
Notice that applying the Cauchy theorem, the expression for $L^s$ coincide with the Dunford-Taylor formula
\beq
L^\beta f = \frac{1}{2\pi i} \int_{\cC}  z^{\beta} R_z(L)fdz,
\label{ldef}
\eeq
where $R_z(L):= (zI - L)^{-1}$ and the integration contour $\cC$ is a
simple curve enclosing the spectrum of $L$ but excluding the negative
real axis and the origin (possible since $c_0 >0$). 
\endcut


\paragraph{Weak Formulation of the Initial Value Problem.} Given a final time $T>0$,
We consider the following weak formulation of \eqref{eq:problem}: find $u\in L^2(0,T; \dH^\beta)$ with $u_t \in L^2(0,T;\dH^{-\beta})$ such that
\begin{equation}\label{eq:weak_pro}
\left\lbrace
\begin{aligned}
  \textcolor{black}{\langle u_t(t),\phi\rangle}+A^\beta(u(t),\phi)&=0 \hbox{ for all } \phi\in\dot{H}^\beta \ \text{and for a.e. } t \in (0,T),\\
  u(0)&=v,  
  \end{aligned}
\right.
\end{equation}
where 
$$ A^\beta(v,w):= (L^\beta v,w) =\sum_{j=1}^\infty \lambda_j^\beta (v,\psi_j) \overline{
  (w,\psi_j)}.$$
From the above definition it follows that $A^\beta$ is a Hermitian sesquilinear form satisfying
$$A^\beta(w,w)=\|w\|_{\dH^\beta}^2.
$$
The standard analysis for time dependent parabolic equations, see for
example \cite{evans}, \textcolor{black}{implies the existence and uniqueness of a solution $u$
to \eqref{eq:weak_pro}.}
(it is the limit of the partial sums below)
and hence   
\begin{equation}\label{e:serie_sol}
u(t) = \sum_{j=1}^\infty  e^{-t\lambda_j^\beta}(v,\psi_j) \psi_j.
\end{equation}
In addition, Cauchy's theorem applied to the partial sums and the
Bochner integrability of the Dunford-Taylor integral below implies
that
\begin{equation}\label{e:dt}
u(t)=e^{-tL^\beta}v:=  \frac 1 {2\pi i} \int_\cC e^{-tz^\beta} R_z(L)v\, dz.
\end{equation}
Here $R_z:=(zI-L)^{-1}$, $z^\beta:= e^{\beta \ln(z)}$ with the logarithm defined with branch
cut along the negative real axis and $\cC$ is a Jordan curve separating
the spectrum of $L$ and the imaginary axis oriented to have the spectrum
of $L$ to its right (see, \eqref{hyper-curve} or \eqref{e:contour} below). 

\paragraph{Dotted Spaces and Sobolev Spaces.}
To understand the approximation properties of finite element methods,
we need to  characterize the spaces $\dH^s$ in terms of Sobolev spaces.
For any $-1\le s \le 2$, set
$$\widetilde{H}^s(\Omega):=\left\{\bal &H^1_0(\Omega)\cap H^s(\Omega),&  1\le s\le 2,\\
				  &[L^2(\Omega),H^1_0(\Omega)]_s,&  0 \le s \le 1,\\
				  &[H^{-1}(\Omega),L^2(\Omega)]_{1+s},& -1\le s \le 0,\eal \right.$$
where $[\cdot,\cdot]_s$ denotes interpolation using the real method. 
By Proposition 4.1 of \cite{frac14}, the spaces $\widetilde{H}^s(\Omega)$ and $\dot{H}^s$ coincide for $s\in[-1,1]$ and their norms are equivalent.

We can consider  $T$ as acting on
 $H^{-1}(\Omega)=\dH^{-1}$ by defining $w:=TF$ as the
solution to \eqref{laxm} with right hand
side replaced by $\langle F,\theta \rangle$. This is an extension of 
the previously defined $T$ upon identifying $L^2(\Omega)$ with its dual.
We then assume:
\begin{enumerate}[(a)]
 \item  There exists $\alpha\in(0,1]$ such that $T$ is a bounded map of $\widetilde{H}^{-1+\alpha}(\Omega)$ into $\widetilde{H}^{1+\alpha}(\Omega)$.
 \item  The functional $F$ defined by 
$$\langle F,\theta\rangle := A(u,\theta),\Forall \theta\in
H^1_0(\Omega)$$
is a bounded operator from $\widetilde{H}^{1+\alpha}(\Omega)$ to 
$\widetilde{H}^{-1+\alpha}(\Omega)$.
\end{enumerate}
Assumptions (a) and (b) implies (see Proposition 4.1 in \cite{frac14}) that  the spaces $\widetilde{H}^s(\Omega)$ and $\dot{H}^s$ coincide for $s\in[-1,1+\alpha]$.

\section{Finite Element Approximation}\label{FEapprox} 

\paragraph{Finite Element Spaces.}
We now consider finite element approximations to \eqref{eq:weak_pro}. 
Let $\{\mathcal{T}_h\}_{h>0}$ be a sequence of globally quasi uniform 
conforming subdivisions of $\Omega$ made of simplexes, i.e. there are positive
constants  $\rho$ and $c$ independent
of $h$ 
such that if $R_\tau$ denotes the diameter of $\tau$ and $r_\tau$
denotes the radius of the largest ball which can be inscribed in $\tau$,
then for all $h>0$,
\begin{equation}\label{ineq:quasi}\bal
R_\tau/r_\tau &\le c\Forall \tau\in \mathcal{T}_h,\quad \hbox{and,}\\
  \max_{\tau\in\mathcal{T}_h}R_\tau &\le \rho\min_{\tau\in\mathcal{T}_h}
  R_\tau.
\eal
\end{equation}
For $h>0$, we
denote $H_h\subset H_0^1$ to be the space of continuous piecewise linear finite element functions
with respect to $\mathcal{T}_h$ and $M$ to be the dimension of $H_h$. 

\paragraph{ The approximations $T_h$ and $L_h$.}
For any $F\in H^{-1}(\Omega)$, we define the finite element approximation $T_h F \in H_h$ of $TF \in H^1_0(\Omega)$ as the unique solution (invoking Lax-Milgram) to
$$A(T_h F,\phi_h)=\langle F,\phi_h\rangle,  \Forall \phi_h\in H_h .$$
Notice that for $f\in L^2(\Omega)$,
$T_hf=T_h \pi_h f$, where $\pi_h$ is the $L^2$-projection onto $H_h$.
We also define  $L_h\ :\ H_h \rightarrow H_h$ by
$$
(L_h v_h,\phi_h)=A(v_h,\phi_h), \Forall \phi_h\in H_h 
$$
and note that $L_h$ is the inverse of $T_h$ restricted to $H_h$.
Similar to $T$, $T_h|_{H_h}$ has positive 
eigenvalues $\{\mu_{j,h}\}_{j=1}^M$  with corresponding 
$L^2$-orthonormal eigenfunctions $\{\psi_{j,h}\}_{j=1}^M$. 
The eigenvalues of $L_h$ are denoted
by $\lambda_{j,h}:=\mu_{j,h}^{-1}$
 for $j=1,2,\cdots,M$.

Furthermore, we define $L_h^\beta\ :\ H_h \rightarrow H_h$ by
$$
L_h^\beta v_h:=\sum_{j=1}^M \lambda_{j,h}^\beta (v_h,\psi_{j,h})\psi_{j,h},$$
and the sesquilinear form $A^\beta_h(\cdot,\cdot)$ on $H_h \times H_h$ by
$$A^\beta_h(v_h,w_h):=\sum_{j=1}^M \lambda_{j,h}^\beta c_{j,h}\bar d_{j,h},\Forall v_h, w_h\in H_h.$$
Here
$$v_h:=\sum_{j=1}^M c_{j,h} \psi_{j,h}\qquad \hbox{and}\qquad
w_h:=\sum_{j=1}^M d_{j,h} \psi_{j,h}.
$$
We also define for $s\in [0,1]$, the discrete dotted norms $\|\cdot\|_{\dot{H}_h^s}$ by
$$
\|v_h\|_{\dot{H}_h^s}:=\Bigg(\sum_{j=1}^M\lambda_{j,h}^s \textcolor{black}{|(v_h,\psi_{j,h})|}^2\Bigg)^{1/2},\qquad \text{for } v_h\in H_h .
$$
It is well known (see e.g., Appendix A.2 in \cite{00BZ}) that for $s\in[0,1]$, there exists a constant $c$ independent of $h$ such that for all  $v_h\in H_h$,
\begin{equation}\label{ineq:H_h_H}
  \frac{1}{c}\|v_h\|_{\dot{H}_h^s}\le\|v_h\|_{\dot{H}^s}\le c\|v_h\|_{\dot{H}_h^s} .
\end{equation}

\paragraph{The Finite Element Approximation of the Initial Value Problem.}
The finite element approximation of \eqref{eq:weak_pro} considered here reads: find 
$u_h\in H^1(0,T;H_h)$ such that for 
  $t\in (0,T)$,
\begin{equation}\label{eq:weakfe}
\left\lbrace
\begin{aligned}
  (u_{h,t}(t),\phi_h)+A^\beta_h(u_h(t),\phi_h)&=0, \forall \phi_h\in H_h\text{, and}\\
  u_h(0)&=v_h:=\pi_h v.
\end{aligned}
\right.
\end{equation}
As in Section~\ref{preli}, the solution of \eqref{eq:weakfe} is given by 
\begin{equation}\label{eq:num_int}
\textcolor{black}{u_{h,t}=}e^{-tL^\beta_h}v_h :=\sum_{j=1}^M e^{-t\lambda_{j,h}^\beta} (v_h,\psi_{j,h})
\psi_{j,h} =
\frac{1}{2\pi i}\int_\cC e^{-tz^\beta }R_z(L_h)v_h\, dz .
\end{equation}
\textcolor{black}{
Note that $\lambda_{h,j} >0$, so that $e^{-t\lambda_{j,h}^\beta}<1$ and in particular the above finite element method is stable:
$$
\| u_{h,t} (t) \| \leq \| v_h \|.
$$
}
\paragraph{Approximation Results.}
By Lemma~5.1 of \cite{BP2015}, 
there exists a constant $c(s,\sigma)$ independent of $h$ such that for
$s\in [0,1]$ and $s+\sigma\le 2$,
\begin{equation}\label{pih-approx}
\|(I-\pi_h)f\|_{\widetilde H^s(\Omega)} \le c(s,\sigma) h^{\sigma}\|f\|_{\widetilde{H}^{s+\sigma}(\Omega)} .
\end{equation}
In addition, for $s\in [0,1]$,
\begin{equation}\label{pih-bound}
\|\pi_h f\|_{\widetilde {H}^{s}(\Omega)}\le c \|f\|_{\widetilde{H}^{s}(\Omega)}.
\end{equation}
The above estimate follows by definition when $s=0$, from \cite{bankdup,l2_stablity} when $s=1$,
and by interpolation for any intermediate $s$.
In addition, we recall the following result from \cite{frac14}.
\begin{proposition}[Corollary 4.2 of \cite{frac14}]\label{prop:T_error}
  Assume (a) holds, then there is a constant $C$ independent of $h$ such that for all $f\in \dot{H}^{\alpha-1}$
  $$\|(T-T_h)f\|_{\dot{H}^{1-\alpha}}\le C h^{2\alpha}\|f\|_{\dot{H}^{\alpha-1}} .$$
\end{proposition}

\paragraph{Space Discretization Error.}
We now estimate the space discretization error for the initial value problem.

\begin{theorem}[Space Discretization Error for the initial value problem]\label{thm:semi_init}
 Assume that (a) and (b) hold for some $\alpha\in (0,1]$
 and that $v\in \dot{H}^{2\delta}$ for $\delta\ge 0$.
    Then there exists a constant $D(t)$ independent of $h$ such that
    \begin{equation}\label{e:num_rate}
    \|(e^{-tL^\beta}-e^{-tL^\beta_h}\pi_h) v\| \le D(t)h^{2\alpha}\|v\|_{\dot{H}^{2\delta}}
    \end{equation}
    where
    \begin{equation}\label{thm_const}
     D(t)=\left \{\bal C: & \qquad \hbox{if }\alpha<\min(\delta,1)
       , \\
     C\max\{1,\ln(1/t)\}:&\qquad \hbox{if } \alpha=1,\delta\ge
     \alpha\hbox{ or } \alpha<1,\delta=\alpha,\\
     Ct^{-(\alpha-\delta)/{\beta}}:& \qquad \hbox{if }\alpha>\delta\ge 0.
    \eal \right. 
\end{equation}
\end{theorem}

\begin{remark}[Asymptotic $t\to0$] If $\alpha<1$ then the above theorem guarantees the rate
  of $h^{2\alpha}$ for all $\delta>\alpha$ without any degeneration as
  $t\rightarrow 0$.
Note that the theorem only guarantees a rate of
of $C \ln(1/t)  h^2$ for small $t$ when  $\delta \ge 1$ and $\alpha=1$.   
In contrast, the classical analysis when $\beta=\alpha=\delta=1$
\cite{thomee} provides the rate $C h^2$ (without the $\ln(1/t)$ for
small $t$).
\end{remark}

Before proving the theorem, we introduce the following lemma whose proof
is postponed until after that of the theorem.

\begin{lemma} There is a positive constant $C(s)$ depending only on 
  $s\in [0,1]$
such that 
\beq|z|^{-s} \|T^{1-s} (z^{-1}I-T)^{-1}f\|\le C(s)\|f\|,\Forall z\in
\cC,f\in L^2(\Omega).
\label{lemeq}
\eeq
The same inequality holds on $H_h$ with $T$ replaced by $T_h$.
\label{tz-bound}
\end{lemma}

We now provide the proof of Theorem~\ref{thm:semi_init} which relies on the following contour $\cC$ in the Dunford-Taylor representations \eqref{e:dt} and \eqref{eq:num_int}:
Given $r_0\in (0,\lambda_1)$, it 
consists in three segments:
\begin{equation}\label{e:contour}
\begin{aligned}
\cC_1&:=\left\lbrace z(r):=re^{-i\pi/4}  \text{ with  }r  \text{ real going from }+\infty \text{ to } r_0\right\rbrace \text{ followed by } \\
\cC_2&:=\left\lbrace z(\theta):=r_0 e^{i\theta} \text{ with }\theta \text{ going from }-\pi/4 \text{ to } \pi/4\right\rbrace \text{ followed by } \\
\cC_3&:=\left\lbrace z(r):=re^{i\pi/4} \text{ with } r \text{ real going from } r_0 \text{ to }+\infty \right\rbrace.
\end{aligned}
\end{equation}

\begin{proof}[Proof of Theorem~\ref{thm:semi_init}.]
The continuous embedding $\dot{H}^{s}\subset\dot{H}^{t}$ when $s>t\ge 0$
implies that it suffices to prove the estimates of the theorem when
\beq\bal
(i) \ \alpha&<\delta\le (1+\alpha/2),\\
(ii)\  \delta&=\alpha \hbox{ and } \alpha \in (0,1],\hbox{ and }\\
(iii)\ \alpha&>\delta\ge 0.
\eal
\label{threeC}
\eeq
We write
$$
\|(e^{-tL^{\beta}}-e^{-tL^{\beta}_h}\pi_h)v\|\le \|(I-\pi_h)e^{-tL^{\beta}}v\| + 
\|\pi_h (e^{-tL^{\beta}}-e^{-tL^{\beta}_h}\pi_h)v\|.
$$
\eat{and first show
\beq
\|(I-\pi_h)(e^{-tL^{\beta}}-e^{-tL^{\beta}_h}\pi_h)v\|\le C(t)
h^{2\alpha} \|v\|.
\label{firstpart}
\eeq
Note that
$$(I-\pi_h)(e^{-tL^{\beta}}-e^{-tL^{\beta}_h}\pi_h)v=(I-\pi_h)e^{-tL^{\beta}}v$$
}

$\boxed{1}$ The approximation property \eqref{pih-approx} of $\pi_h$ immediately yields
$$\|(I-\pi_h)e^{-tL^{\beta}}v\|\leq C
h^{2\alpha}\|e^{-tL^{\beta}} v\|_{\dot{H}^{2\alpha}} .$$
We estimate $\|e^{-tL^\beta}v\|_{\dot{H}^{2\alpha}}$ by expanding $v$ in the basis generated by the eigenfunction of $L$. Let
$c_j:=(v,{\psi}_j)$ be the Fourier coefficient of $v$. We distinguish two
cases. 
When $\delta\ge\alpha$, we use the representation \eqref{e:serie_sol} of $e^{-tL^\beta}v$ to write
$$\|e^{-tL^{\beta}}v\|^2_{\dot{H}^{2\alpha}}=
\sum_{j=1}^{\infty}\lambda_j^{2\alpha}\underbrace{e^{-2 \lambda_j^\beta}}_{\leq 1} |c_j|^2
\le \lambda_1^{2(\alpha-\delta)}
\sum_{j=1}^{\infty}\lambda_j^{2\delta}|c_j|^2=\lambda_1^{2(\alpha-\delta)}
\|v\|_{\dot{H}^{2\delta}}^2.$$
Otherwise, when $\delta < \alpha$,
$$
\|e^{-tL^{\beta}} v\|_{\dot{H}^{2\alpha}}^2 
=t^{-2(\alpha-\delta)/{\beta}}\sum_{j=1}^\infty
\lambda_j^{2\delta}|(t\lambda^\beta_j)^{(\alpha-\delta)/{\beta}}  e^{-t \lambda_j^{\beta}}|^2 |c_j|^2\leq
Ct^{-2(\alpha-\delta)/{\beta}}
\|v\|^2_{\dot{H}^{2\delta}},$$
where for the last inequality we used that $x^\eta e^{-x}\le
C(\eta)=C $ for $x\ge 0$ and $\eta=(\alpha-\delta)/\beta$.

$\boxed{2}$ We are now left to bound
\beq 
\|\pi_h(e^{-tL^{\beta}}-e^{-tL^{\beta}_h}\pi_h)v\|.
\label{tobound}
\eeq
The Dunford-Taylor integral representation gives
$$\pi_h(e^{-tL^{\beta}}-e^{-tL^{\beta}_h}\pi_h)v=\frac{1}{2\pi
  i}\int_{\cC}e^{-tz^\beta}\pi_h(R_z(L)-R_z(L_h)\pi_h)v \, dz$$
from which we deduce that
$$\|\pi_h(e^{-tL^{\beta}}-e^{-tL^{\beta}_h}\pi_h)v\|\leq \frac{1}{2\pi}\int_{\cC}|e^{-tz^\beta }| \|\pi_h(R_z(L)-R_z(L_h)\pi_h)v\|\, d|z| .$$
Noting that $(z-L)^{-1}=T(zT-I)^{-1}$ and 
 $(z-L_h)^{-1}=(zT_h-I)^{-1}T_h$, we obtain 
 \textcolor{black}{
\begin{align*}
\pi_h(R_{z}(L)-R_{z}(L_h)\pi_h)&=\pi_h((z-L)^{-1}-(z-L_h)^{-1})\\
&=\pi_h(T(z T-I)^{-1}-(z T_h-I)^{-1}T_h)\\
&=\pi_h(zT_h-I)^{-1}((zT_h-I)T-T_h(zT-I))(zT-I)^{-1}\\
&=\pi_h(zT_h-I)^{-1}(T_h-T)(zT-I)^{-1}\\
&=-(z T_h -I)^{-1}\pi_h(T-T_h)(z T-I)^{-1},
\end{align*}
where for the last step we used the fact that $\pi_h(zT_h-I)^{-1}=(zT_h-I)^{-1}=(zT_h-I)^{-1}\pi_h$.}
Whence,
$$
\|\pi_h(e^{-tL^{\beta}}-e^{-tL^{\beta}_h}\pi_h)v\|\le 
C\int_{\cC}|e^{-tz^\beta}| |z|^{-1+\alpha-\delta} \|W(z)\| \, d|z|
$$
with
$$W(z):= |z|^{1-\alpha+\delta} (z T_h -I)^{-1}\pi_h(T-T_h)(z T-I)^{-1}v.
$$
To complete the proof, we show first that
\beq 
\|W(z)\| \le C h^{2\alpha} \|v\|_{\dH^{2\delta}}
\label{w-bound}
\eeq
for a constant $C$ independent of $h$ and $z$ 
and then that
\beq 
\int_{\cC}|e^{-tz^\beta}| |z|^{-1+\alpha-\delta}  \, d|z|
\le  D(t).
\label{i-bound}
\eeq

$\boxed{3}$ We start with \eqref{w-bound}.  
Rewriting
$$
(z T_h -I)^{-1}\pi_h(T-T_h)(z T-I)^{-1}=z^{-2}( T_h
-z^{-1})^{-1}\pi_h(T-T_h)( T-z^{-1})^{-1},
$$
we deduce that
\beq\bal
\|W(z)\|&\le\underbrace{\|z^{-(1+\alpha)/2}( T_h
-z^{-1})^{-1}\pi_h\|_{\dot{H}^{1-\alpha}\rightarrow
  L_2(\Omega)}}_{:=\mathrm{I}}\underbrace{\|(T-T_h)\|_{\dot{H}^{\alpha-1}\rightarrow\dot{H}^{1-\alpha}}}_{:=\mathrm{II}}\\
&\underbrace{\|z^{-(1+\alpha)/2+\delta}(
T-z^{-1})^{-1}\|_{\dot{H}^{2\delta}\rightarrow\dot{H}^{\alpha-1}}}_{:=\mathrm{III}}.
\eal
\label{threet}
\eeq
Now, we estimate the three terms in the right hand side separately.
For $\mathrm{III}$ we have
$$\bal
\|(T-z^{-1})^{-1}\|_{\dot{H}^{2\delta}\rightarrow\dot{H}^{\alpha-1}}
&=\sup_{w\in \dot{H}^{2\delta}} \frac {\|T^{(1-\alpha)/2} (T-z^{-1})^{-1} w\|}
{\|L^\delta w\|}\\
& = \sup_{\theta\in L^2(\Omega)} \frac {\|T^{(1-\alpha)/2} (T-z^{-1})^{-1}
  T^\delta \theta\|}
{\|\theta\|} =\|T^{1-[(1+\alpha)/2-\delta]} (T-z^{-1})^{-1}\|.
\eal
$$
We see that for all three cases \eqref{threeC},
$0\le s:=(1+\alpha)/2-\delta\le 1$ 
so that Lemma~\ref{tz-bound} applies and
\beq
 \mathrm{III} = \|z^{-(1+\alpha)/2+\delta}(
T-z^{-1})^{-1}\|_{\dot{H}^{2\delta}\rightarrow\dot{H}^{\alpha-1}}\le 
C.
\label{thirdt}
\eeq

To estimate $\mathrm{I}$, we invoke  \eqref{pih-bound} and  \eqref{ineq:H_h_H} to write
\begin{equation}
\begin{split}
\textcolor{black}{ \|( T_h
-z^{-1})^{-1}\pi_h \|_{\dot{H}^{1-\alpha}\rightarrow  L_2}}
&\le \|( T_h-z^{-1})^{-1}\|_{\dot{H}_h^{1-\alpha}\rightarrow L_2}
\|\pi_h\|_{\dot{H}^{1-\alpha}\rightarrow \dot{H}_h^{1-\alpha}} \\
&\leq C \|( T_h-z^{-1})^{-1}\|_{\dot{H}_h^{1-\alpha}\rightarrow L_2}\\
&\leq C \|T_h^{1-(1+\alpha)/2}( T_h-z^{-1})^{-1}\|.
\end{split}
\label{firstt}
\end{equation}
Applying Lemma~\ref{tz-bound} again gives
\beq
\mathrm{I} = \|z^{-(1+\alpha)/2}
(T_h
-z^{-1})^{-1}\pi_h \|_{\dot{H}_h^{1-\alpha}\rightarrow  L_2}
\le C.
\label{firstfinal}
\eeq
Combining \eqref{thirdt}, \eqref{firstfinal} and applying 
Proposition~\ref{prop:T_error} to estimate $\mathrm{II}$ yield \eqref{w-bound}.

$\boxed{4}$ We finally prove \eqref{i-bound}.  Note that $|z|= r_0$ for $z\in\cC_2$ and hence 
$$\int_{\cC_2}|e^{-tz^\beta}| |z|^{-1+\alpha-\delta}  \, d|z|\le C.$$
For the remaining part of the contour, we have
\beq
\mathscr{I}:=\int_{\cC_1\cup \cC_3}|e^{-tz^\beta}| |z|^{-1+\alpha-\delta}  \, d|z|
= 2 
\int_{r_0}^{\infty}e^{-\cos(\beta\pi/4)tr^\beta}r^{-1+\alpha-\delta}dr.
\label{intbound}
\eeq
If $\delta>\alpha$
$$\mathscr{I}\le C\int_{r_0}^\infty r^{-1+\alpha-\delta}\, dr\le C .$$
If $\delta<\alpha$, making the change of variable
$y=\cos(\beta\pi/4)tr^\beta$ gives
$$\mathscr{I}=Ct^{(\delta-\alpha)/\beta}\int_{\cos(\pi \beta/4)tr_0^\beta}^\infty
e^{-y}y^{-1+(\alpha-\delta)/\beta}\,dy\le \frac{C}{\beta}t^{(\delta-\alpha)/\beta}\Gamma\Big(\frac{\alpha-\delta}{\beta}\Big)$$
where $\Gamma(x)$ is the Gamma function.  Finally, when $\delta=\alpha$,
the same change of variables gives 
$$\mathscr{I}=C\int_{\cos(\pi \beta/4)tr_0^\beta}^\infty
e^{-y}y^{-1}\,dy.$$
If $\cos(\pi \beta/4)tr_0^\beta\ge 1$  then 
$$\mathscr{I}\le C\int_1^\infty e^{-y}y^{-1}\,dy
\le C\int_1^\infty e^{-y}\,dy=C/e.$$
Otherwise, splitting the integral gives
$$\bal
\mathscr{I}&\le C\int_{\cos(\pi
  \beta/4)tr_0^\beta}^1e^{-y}y^{-1}\,dy+C/e\\
&\le C\int_{\cos(\pi \beta/4)tr_0^\beta}^1y^{-1}\,dy+C/e
\le C\max\{1,\ln(1/t)\}.
\eal
$$
This completes the proof of the theorem.
\end{proof}

\begin{proof} [Proof of Lemma~\ref{tz-bound}] 
As the proof of the
  continuous and discrete cases are essentially identical, we only provide 
the proof for the former.  
For $z\in \cC_2$, the triangle inequality
implies
$$r_0^{-1}-\mu_1\le r_0^{-1}-\mu_j \le |z^{-1}-\mu_j|.$$
Also,
$$|z|^{-1}+\mu_j\le 2r_0^{-1}$$
so 
$$\frac {r_0^{-1}-\mu_1}{2r_0^{-1}} (|z|^{-1}+\mu_j) \le
|z^{-1}-\mu_j|.$$

If $z \in \cC_3$ then $z^{-1}$ is on the line connecting 0 to
$r_0^{-1}e^{i\pi/4}$. It follows that $|z^{-1}-\mu_j| \ge \Im(z^{-1})=
|z|^{-1}/\sqrt2$. Also, $|z^{-1}-\mu_j|$ is greater than or equal to the
distance from $\mu_j$ to the line segment, i.e., $|z^{-1}-\mu_j|\ge
\mu/\sqrt2$.  The same inequalities hold for $\cC_1$ and hence
$$|\mu_j-z^{-1}|\geq
\frac{\sqrt{2}}{4}(\mu_j+|z|^{-1}),\Forall z\in \cC_1\cup \cC_3.$$
 Thus, we have shown that for every $z\in \cC$
\beq
|\mu_j-z^{-1}|^{-1} \le C (\mu_j+|z|^{-1})^{-1}.
\label{mubound}
\eeq

Expanding the square of left hand side of \eqref{lemeq} gives
\beq
\mathcal W:=|z|^{-2s}\|T^{1-s} (z^{-1}I-T)^{-1}f\|^2=\sum_{j=1}^\infty
\bigg( \frac {|z|^{-s} \mu_j^{1-s}} {|z^{-1}-\mu_j|}\bigg)^2
|(f,\psi_j)|^2.
\label{lem1i}
\eeq
Thus, \eqref{mubound} implies
$$\mathcal W\le C
\sum_{j=1}^\infty
\bigg( \frac {|z|^{-s} \mu_j^{1-s}} {|z|^{-1}+\mu_j}\bigg)^2
|(f,\psi_j)|^2.
$$
A Young's inequality yields
$$\frac {|z|^{-s} \mu_j^{1-s}} {|z|^{-1}+\mu_j} \leq 
\frac{s  |z|^{-1} +(1-s)\mu_j}{|z|^{-1}+\mu_j}\le 1
$$
so that 
$$\mathcal W \le C\|f\|^2 .$$
This completes the proof of the lemma.
\end{proof}

To end this section, we consider the non-honmogeneous parabolic equation
for a given $f \in L^2(0,T;L^2(\Omega))$ but with zero initial data,
i.e., 
find $u \in L^2(0,T; \dot{H}^\beta)$ such that $u_t \in L^2(0,T;\dH^{-\beta})$ and for a.e $t\in (0,T)$
\begin{equation}\label{eq:weak_nonho}
\left\lbrace
\begin{aligned}
  \textcolor{black}{\langle u_t(t),\phi\rangle}+A^\beta(u(t),\phi)&=(f(t),\phi), \qquad\Forall \phi\in\dot{H}^\beta\hbox{ and}\\
  u(0)&=0. 
  \end{aligned}
\right.
\end{equation}
By Duhamel's principle, the solution to \eqref{eq:weak_nonho} is given by
$$u(t)=\int_0^t e^{-(t-s)L^\beta}f(s)\, ds .$$
The corresponding finite element approximation reads: find $u_h\in H^1(0,T;H_h)$ such that
for a.e $t\in (0,T)$
\begin{equation}\label{eq:weakfe_nonho}
\left\lbrace
\begin{aligned}
  (u_{h,t},\phi_h)+A^\beta_h(u_h,\phi_h)&=(f(t),\phi_h), \qquad\forall \phi_h\in H_h\hbox{ and}\\
  u_h(0)&= 0,
\end{aligned}
\right.
  \end{equation}
or 
$$u_h(t)=\int_0^t e^{-(t-s)L_h^\beta}\pi_h f(s)\, ds .$$
Applying Theorem \ref{thm:semi_init} we obtain that
$$
\bal
\|u(t)-u_h (t)\|&\le \int_0^t \|(e^{-sL^\beta}-e^{-sL_h^\beta}\pi_h)f(s)\|\, ds\\
&\le C h^{2\alpha} \int_0^t D(s)\|f(t-s)\|_{\dH^{2\delta}}\, ds ,
\eal
$$
where $D(s)$ is given in \eqref{thm_const}.
Therefore, the optimal convergence rate $2\alpha$ is achieved provided the above integral is finite.
For example,  if $f \in L^{\infty}(0,T;\dot{H}^{2\delta})$, we have the following corollary.
\begin{corollary}[Space discretization for Non-Homeogenous Right Hand Side]\label{cor:nonho_semi}
Assume that (a) and (b) hold for $\alpha\in (0,1]$. 
For $\delta>0$ with $\delta \not = \alpha$ assume furthermore that $f\in L^{\infty}(0,T;\dot{H}^{2\delta})$ and denote 
$u$ and $u_h$ to be  the solutions of
 \eqref{eq:weak_nonho} and \eqref{eq:weakfe_nonho}, respectively.
In addition, for any sufficiently small $\epsilon>0$, set $\overline \alpha := \min(\alpha,\beta+\delta-\epsilon)>0$.
Then, there exists a constant $C$ such that 
 $$\|u(t)-u_h (t)\|\le  D(h,t)\|f\|_{L^{\infty}(0,T;\dot{H}^{2\delta})},$$
where
    \begin{equation}
     D(h,t):=\left \{\bal Cth^{2\alpha}: & \qquad\hbox{if } \alpha<\delta, \\
     Ct^{1-\frac{\overline\alpha-\delta}{\beta}}h^{2\overline \alpha}:& \qquad\hbox{if } \alpha>\delta.\\ 
    \eal \right. 
\end{equation}
 \end{corollary}
 
\section{Quadrature approximation for \eqref{eq:num_int}}
\label{quad}
In this section, we develop exponentially convergent quadrature approximations to 
\eqref{eq:num_int} based on the contour $\cC$ given by
\eqref{hyper-curve}. To this end, we extend $\gamma$ to the complex
plane, i.e.,   
\beq \gamma(z):=b(\cosh{ z}+i\sinh{ z}),\quad z\in \CC.
\label{gamma}
\eeq
We then have
for $w_h\in H_h$, 
\begin{equation}\label{eq:init_sol}
\bal
e^{-tL^{\beta}_h}w_h&=\frac{1}{2\pi i}\int_{\widehat{\cC}}e^{-tz^\beta}R_z(L_h)w_h\, dz\\
&=\frac{1}{2\pi i} \int_{-\infty}^\infty e^{-t \gamma(y)^\beta} \gamma'(y)[(\gamma(y)I-L_h)^{-1}w_h] \, dy .
\eal
\end{equation}
The sinc quadrature approximation $Q_h^{N,k}(t)w_h$ to  \eqref{eq:init_sol} with $2N+1$ quadrature points and quadrature spacing parameter $k>0$ is defined by
\begin{equation}\label{eq:initq}
  Q_h^{N,k}(t)w_h:=\frac{k}{2\pi i}\sum_{j=-N}^{N} e^{-t \gamma(y_j)^\beta}\gamma'(y_j)
[(\gamma(y_j)I-L_h)^{-1}w_h]\  
\end{equation}
with $y_j:=jk$.

Expanding $w_h$ in terms of the discrete eigenvector basis 
$\{\psi_{j,h}\}$ gives
\beq
\bal
\| (e^{-tL_h^\beta}- Q_h^{N,k}(t)) w_h\|^2 &= (2\pi )^{-2}\sum_{j=1}^M |\cE(\lambda_{j,h},t)|^2
|(w_h,\psi_{j,h})|^2\\ & \le 
(2\pi)^{-2} \max_{j=1,\cdots,M} |\cE(\lambda_{j,h},t)|^2 \|w_h\|^2,
\eal
\label{errorl}
\eeq
where 
$$\cE({\lambda},t):=  \int_{-\infty}^{\infty} g_{\lambda}(y,t)\, dy -
 k\sum_{j=-N}^{N}g_{\lambda}(j k,t)$$
and
\begin{equation}\label{eq:lamdafun}
  g_\lambda(z,t)=  
e^{-t\gamma(z)^\beta}(\gamma(z)-\lambda)^{-1}\gamma'(z),\quad \hbox{ for
  } z\in \CC,t>0.
\end{equation} 

The theorem below guarantees the exponential decay of the quadrature error. It uses the following notations for $b \in (0,\lambda_1/\sqrt{2})$, $N$ an integer, $k>0$ as above and  $d\in  (0,\pi/4)$: 
\beq
\bal  
\kappa&:=\cos\big[\beta(\pi/4+d)\big] \bigg[\sqrt 2 b  \sin(\pi/ 4-d)\bigg]^\beta,\\
  N(d,t)&:=\max_{\lambda\ge \lambda_1} \bigg \{\int_{-\infty}^\infty
  |g_\lambda(y+id,t)|+|g_\lambda(y-id,t)|\ dy\bigg \}, \quad \hbox{and}\\
M(t)&:=(1+\cL(\kappa t)),\qquad \hbox{ where }\cL(x):=1+|\ln(1-e^{-x})|.
\eal
\label{kappaetc}
\eeq

\begin{theorem} [Quadrature Theorem] \label{quad_theorem}
For integer $N$ and real number $k>0$, 
let $Q_h^{N,k}(t)$ be the sinc quadrature approximation given by
\eqref{eq:initq}. Then  there is a constant $C$ not depending on $t$, $h$, $k$ and $N$ such that 
 \beq
\|e^{-tL_h^\beta}-Q_h^{N,k}(t)\|_{L^2(\Omega)\rightarrow L^2(\Omega)}
 \leq C\Big(\frac{N(d,t)}
{e^{2\pi d/k}-1}+\frac{M(t)\cosh(kN)} {\sinh(kN)} 
e^{-\kappa 2^{-\beta} t e^{kN\beta}}
\Big).
\label{Therror}
\eeq
The function $N(d,t)$ is uniformly bounded  when $t$ is bounded
away from zero and bounded by $CM(t)$ 
as $t\rightarrow 0$.
\end{theorem}

We refer to Examples~\ref{Example_exp} and \ref{Example_exp_balance} for a discussion on the relation between $N$ and $k$
and, in particular, how to get uniform bounds on $\cosh(kN)/\sinh(kN)$. 
Theorem~\ref{quad_theorem} together with the space discretization error estimate provided by Theorem~\ref{thm:semi_init} implies the following result about the total error.
\begin{corollary}[Total Error for the initial value problem]\label{c:total_error}
 Assume that (a) and (b) hold for $\alpha\in (0,1]$ and $\delta $ is
 nonnegative with $ \delta \not = \alpha$. Let $D(t)$ be the constant in \eqref{thm_const} and
 $N(d,t)$, $M(t)$, and $\kappa$ be as in \eqref{kappaetc}. Then, there exists a constant $C$ independent of $t$, $h$, $k$ and $N$
 such that
 \beq\label{total_esti}\bal
    \|(e^{-tL^\beta}-Q_h^{N,k}(t)\pi_h) v\|& \le
    D(t)h^{2\alpha}\|v\|_{\dot{H}^{2\delta}}\\
&\quad + C\Big(\frac{N(d,t)}
{e^{2\pi d/k}-1}+\frac{M(t)\cosh(kN)} {\sinh(kN)} 
e^{-\kappa 2^{-\beta} t e^{kN\beta}}\Big)\|v\|.
\eal
\eeq
\end{corollary}
Before proving Theorem~\ref{quad_theorem}, we mention a fundamental ingredient taken from \cite{sinc_int}.

\begin{lemma} [Lemma 1 of \cite{sinc_int}]   \label{lem1} 
For $r,s>0$, 
\beq
\int_0^\infty e^{-s \cosh(x)}\, dx \le \cL(s)
\label{Ibound1}
\eeq
and
\beq 
\int_r^\infty e^{-s \cosh(x)}\, dx \le (1+\cL(s)) e^{-s
    \cosh r},
\label{Ibound2}
\eeq
with $\cL(s)$ given in \eqref{kappaetc}
\end{lemma}

\begin{proof} [Proof of Theorem~\ref{quad_theorem}] Fix $t>0$.
In view of \eqref{errorl} and since $\lambda_{1,h}>\lambda_1$, it suffices to show that $|\cE(\lambda,t)|$
is bounded by the right hand side of \eqref{Therror} for $\lambda \ge
\lambda_1$.   To prove this, we follow \cite{lund1992sinc}.
We have 
\beq
|\cE({\lambda},t)|\le  \bigg| \int_{-\infty}^{\infty} g_{\lambda}(y,t)\, dy -
 k\sum_{j=-\infty}^{\infty}g_{\lambda}(j k,t)\bigg |+k \sum_{|j|>N}| g_{\lambda}(j k,t)|.
\label{splitsum}
\eeq

To bound the first term on the right hand side of \eqref{splitsum}, 
we apply standard estimates for sinc quadratures \cite{lund1992sinc}
\beq\bal
\bigg|\int_{-\infty}^\infty g_\lambda(y,t)\, dy -\sum_{j=-\infty}^{\infty}
g_\lambda (kj)\bigg|&\le \frac {N(d,t)}{2\sinh(\pi d/k)}
e^{-\pi d/k}\\ & \quad =
\frac {N(d,t)}
{e^{2\pi d/k}-1}
,\eal
\label{sinc1}
\eeq
which are valid provided: 
\begin{enumerate} [(i)]
\item For each $\lambda\ge \lambda_1$ and $t>0$, 
$g_\lambda(z,t)$ is an analytic function  of $z$ on the strip 
$$
S_d:=\{z\ :\ \Im(z)<d\};
$$
\item $\displaystyle
  \int_{-d}^d |g_\lambda(y+i\eta,t)|\ d\eta\le C,\Forall y\in \RR$;
\item $N(d,t)<\infty$ for $t>0$.
\end{enumerate}

We now show that the conditions are satisfied. Note that for $z\in \CC$,
\beq\bal
\Re(\gamma(z))&=\sqrt 2 b \cosh(\Re(z)) \sin\bigg(\frac
\pi4-\Im(z)\bigg)\quad\hbox{and}
\\
\Im(\gamma(z))&=\sqrt 2 b \sinh(\Re(z)) \sin\bigg(\frac
\pi4+\Im(z) \bigg).
\eal
\label{gamids}
\eeq
It follows that $\Re(\gamma(z))>0$ for $z\in \bar S_d=\{z\ :\ \Im(z)\le d\}$.
Thus, to prove (i), it suffices to show that $\lambda$ is not contained in the image of $S_d$ under $\gamma$.  
In fact, we shall show that 
there is a constant $C>0$ such that 
\beq
|(\gamma(z)-\lambda)|\ge C  \Forall z\in S_d\hbox{ and }\lambda\ge
\lambda_1 .
\label{glambdai}
\eeq

To see this, let $y_0 >0$ be any number  such that
$C_0:=\lambda_1-\sqrt{2}b\cosh{(y_0)}>0$.
Then for $z\in \overline {S}_d$ with $|\Re(z)|\le y_0$, 
\beq
\bal
|\gamma(z)-\lambda|&\geq|\Re(\gamma(z)-\lambda)|
=|\sqrt{2}b\cosh{(\Re(z))}\sin\bigg(\frac{\pi}{4}-\Im(z)\bigg) -\lambda|\\
&\geq \lambda_1-\sqrt2b\cosh(y_0)=C_0.
\eal
\label{gb1}
\eeq
On the other hand, if  $z\in \overline {S}_d$ with $|\Re(z)|> y_0$,
\beq\bal
|\gamma(z)-\lambda|\geq|\Im(\gamma(z)-\lambda)|
&=\sqrt{2}b|\sinh{(\Re(z))}|\sin\bigg(\Im(z)+\frac{\pi}{4}\bigg)\\
&\ge \sqrt2 b \sinh(y_0) \sin\bigg(\frac\pi 4 - d\bigg).
\eal
\label{gb2}
\eeq
Combining the above two inequalities shows 
\eqref{glambdai} and hence (i).   

To verify (ii) and (iii), we provide bounds on
$|g_\lambda(z,t)|$ for $z\in {\overline S}_d$ and $\lambda \ge \lambda_1$.
Similar computations leading to \eqref{gamids} implies that
\beq
|\gamma^\prime(z)| \le |\Re{(\gamma'(z))}|+|\Im{(\gamma'(z))}|
 \le 2\sqrt 2 b \cosh(\Re(z)).
\label{gpbound}
\eeq
For $z\in \overline {S}_d$ with  $|\Re(z)|\le y_0$, 
\eqref{gb1} yields
\beq
|g_\lambda(z,t)|\le 
\frac {2\sqrt2b \cosh(y_0)}{C_0}
\big|e^{-t\gamma(z)^\beta}\big|.
\label{glb1}
\eeq
Similarly, for $z\in {\overline S}_d$ with $|\Re(z)|> y_0$, there holds
\beq
|g_\lambda(z,t)|\le 
 \frac {2\cosh(\Re(z))} 
{|\sinh(\Re(z))|\,\sin(\pi/4-d)}
|e^{-t\gamma(z)^\beta}|
\le C|e^{-t\gamma(z)^\beta}|,
\label{glb2}
\eeq
where to derive the last inequality, we used
$$\frac {\cosh(x)} 
{|\sinh{(x)}|}\le \bigg|1+\frac{2}{e^{2y_0}-1}\bigg|, \quad \hbox{ for }
x\in \RR \hbox{ with } |x|\ge y_0.$$

We next bound the exponential on the right hand side of \eqref{glb1} and 
\eqref{glb2}.  To this end, we note that by \eqref{gamids},
$$\bigg|\frac {\Im(\gamma(z))} {\Re(\gamma(z))}\bigg | =
\frac {|\sinh(\Re(z))| \sin(
\pi/4+\Im(z) )}
{\cosh(\Re(z)) \sin(
\pi/ 4-\Im(z))}\le \tan(\frac{\pi}{4}+d),\Forall z\in {\overline S}_d.$$
Thus, 
$$|\arg(\gamma(z))| \le \frac{\pi}{4}+d,\Forall z\in {\overline S}_d,$$
so that together with the observation  $|\gamma(z)|\ge |\Re(\gamma(z))|$ and
\eqref{gamids}, we arrive at
$$\bal
\Re(\gamma(z)^\beta)&= |\gamma(z)|^\beta \cos(\beta \arg(\gamma(z))
\\&\ge \cos(\beta (\pi/4+d)) |\gamma(z)|^\beta
\\&\ge\cos(\beta  (\pi/4+d)) |\Re(\gamma(z))|^\beta\\
&\ge \kappa \cosh(\Re(z))^\beta,\Forall z\in {\overline S}_d.
\eal
$$

Combining the above inequality with \eqref{glb1} and \eqref{glb2} 
shows that 
\beq
|g_\lambda(z,t)|\le C e^{-t \Re(\gamma(z)^\beta)}\le 
Ce^{-t\kappa \cosh(\Re(z))^\beta}, \Forall z\in {\overline S}_d\hbox{ and }
\lambda\ge \lambda_1.
\label{finalglb}
\eeq
This immediately implies (ii), i.e.
$$  \int_{-d}^d |g_\lambda(y+i\eta,t)|\ dy \le
2d Ce^{-t\kappa (\cosh \ y)^\beta} \le C.$$
To show (iii), we use again \eqref{finalglb} to deduce
\beq
\bal
N(d,t)&=\max_{\lambda\geq \lambda_1}{\int_{-\infty}^{\infty}(|g_{\lambda} (y-id,t)|+|g_{\lambda}(y+id,t)|)\, dy}\\
&\leq C\int_0^{\infty} e^{-\kappa t (\cosh{y})^\beta}\, dy\\ &\le 
C\int_0^{1} e^{-\kappa t (\cosh{y})^\beta}\, dy\ +C\int_{1}^{\infty} e^{-\kappa t (\cosh{y})^\beta}\, dy.
\eal
\label{N-bound}
\eeq
The first integral is bounded by $1$.  For the second,
making the change of integration variable, $(\cosh y)^\beta=\cosh u$,
gives
\beq
I_2:=\int_1^{\infty} e^{-\kappa t (\cosh{y})^\beta}\, dy
=\frac{1}{\beta}\int_{u_0}^\infty e^{-\kappa
  t \cosh u } \frac{\sinh u 
  \cosh y}{\cosh{u}\sinh y} \,du
\label{1bound}
\eeq 
where $u_0=\cosh^{-1}[(\cosh{(1)})^\beta]$.
As $\cosh(y)/\sinh(y)$ is decreasing for positive $ y$ 
and $\sinh (u)/\cosh(u) <1$ for positive $u$, 
applying Lemma~\ref{lem1} gives
\beq
I_2 \le \frac {\cosh(1)} {\beta  \sinh(1)}
\int_{u_0}^\infty e^{-\kappa
  t \cosh u }\, du 
\le \frac {\cosh(1)} {\beta  \sinh(1)}
(1+\cL(\kappa t)) e^{-\kappa t 
\cosh(1)^\beta}.
\label{i2bound}
\eeq
Combining this with the bound for the first integral of the right hand side of \eqref{N-bound} 
 proves (iii) and concludes the estimation for the first term in \eqref{splitsum}.

For the second term of \eqref{splitsum}, we again apply \eqref{finalglb}
and obtain
\beq\bal J:&= \bigg |k \sum_{|j|>N} g_{\lambda}(j k,t)\bigg|
\le Ck \sum_{|j|>N} e^{-t\kappa \cosh(jk)^\beta}
\le C \int_{kN}^\infty e^{-t\kappa \cosh( y)^\beta}\, dy.
\eal
\label{stsplit}
\eeq
Repeating the arguments in \eqref{1bound} and \eqref{i2bound} 
(with $u_0:=\cosh^{-1}(\cosh(kN)^\beta)$) gives
$$\bal
J 
&\le \frac {C \cosh(kN)} { \sinh(kN)}
(1+\cL(\kappa t)) e^{-\kappa t \cosh(kN)^\beta}\\
&\le \frac {C \cosh(kN)} { \sinh(kN)}
(1+\cL(\kappa t)) e^{-\kappa 2^{-\beta} t e^{kN\beta}}
.\eal
$$
This completes the estimate for the second term of \eqref{splitsum} and proof.
\end{proof}

We conclude this section with two examples illustrating different choices of $k$
and $N$.

\begin{example}[Large t] \label{Example_exp}
 As in \cite{sinc_int}, we set $k:=\ln N/(\beta N)$ for some $N>1$.  
The mononicity of
  $\cosh x/\sinh x $ for positive $x$ implies that
 \beq
\frac {\cosh(kN)} { \sinh(kN)}\le \frac {\cosh(\ln 2/\beta)} 
{\sinh(\ln 2/\beta)}
\label{u1est}
\eeq
and hence
\beq
\|e^{-tL_h^\beta}-Q_h^{N,k}(t)\|_{L^2(\Omega)\rightarrow L^2(\Omega)}
 \leq C\Big(\frac{N(d,t)}
{e^{{ 4\beta\pi d N}/{\ln{N}}}-1}+\frac{M(t)} {e^{\kappa t2^{-\beta}N}}\Big)=O(e^{-CN/\ln N}).
\label{first_ex}
\eeq
\end{example}

\begin{example}[Small t] \label{Example_exp_balance}
When $t$ is small, we attempt to balance the error coming
  from the two terms in \eqref{Therror} by setting
$$\frac {2\pi d }{k}\approx   \kappa t 2^{-\beta} e^{\beta N k}.$$
To this end, given an integer 
$N>0$, we define $k$ to be the unique positive solution
of 
$$ke^{\beta Nk}= \frac {2^{1+\beta}\pi d} {\kappa t}.$$
Hence, for $t\le 1$ and $N>1$,
$$Nk\ e^{\beta Nk} = \frac {2^{2+\beta}N \pi d} {\kappa t}
\ge \frac {2^{1+\beta}\pi d} {\kappa }.$$
In particular $Nk\ge \zeta$ where $\zeta$ is the positive solution of 
$$\zeta e^{\beta \zeta} =\frac  {2^{1+\beta}\pi d} {\kappa }$$
so that the monotonicity of $\cosh(\cdot)/\sinh(\cdot)$ implies
$$\frac {\cosh(kN)} { \sinh(kN)}\le \frac {\cosh(\zeta)} 
{\sinh(\zeta)}.$$
As a consequence, we obtain the quadrature error estimate
$$
\|e^{-tL_h^\beta}-Q_h^{N,k}(t)\|_{L^2(\Omega)\rightarrow L^2(\Omega)}
 \leq C({N(d,t)} +{M(t)}) e^{-2\pi d/k}
\text{.}
$$
\end{example}

\section{Numerical Illustrations}\label{numerical}
In this section, we present some numerical experiments illustrating the error estimates provided in Sections~\ref{FEapprox} and 
\ref{quad}. 
\subsection{The Effect of the Space Discretization.\\}
\paragraph{A One Dimensional Initial Value Problem.}
Consider the one-dimensional domain $\Omega:=(0,1)$,
$Lu:=-u^{\prime\prime}$, $f\equiv 0$ and the initial condition
\beq
v(x):=\left \{\bal 2x, & \qquad x<0.5,\\
     2-2x,& \qquad x\geq 0.5.
         \eal \right.
\label{v(x)}
\eeq
Note that $v$  of \eqref{v(x)} 
belongs to $\dot{H}^{\frac32-2\epsilon}(0,1)$ for any $\epsilon>0$.
Theorem~\ref{thm:semi_init} with $\alpha=1$ guarantees 
$$\| u(t) - u_h(t) \| \leq C t^{-\frac{1/4+\epsilon}{\beta}} h^2.$$

To illustrate the error behavior predicted by  \eqref{e:num_rate}, we
use a mesh of equally spaced points, i.e., $h=1/(M+1)$ with $M$ being
the number of interior nodes.  We set $H_h$ to be the set of continuous
piecewise linear functions with respect to this mesh vanishing at 0 and
1.   The resulting stiffness and mass matrices, denoted by  $\mathcal A_h$ and 
$\mathcal M_h$,  
 are defined in terms of 
the standard hat-function finite element basis $\{\phi_i\}$, $i=1,\ldots
M$ and are tri-diagonal matrices
with tri-diagonal entries $h^{-1}(-1,2,-1)$ and $h(1/6,4/6,1/6)$,
respectively. The operator $L_h$ is given by $L_h=\mathcal M_h^{-1}
\mathcal A_h$.

\def\mcD{{\mathcal D}}
These matrices can be diagonalized using the discrete sine
transform, i.e., the $M\times M$ matrix with entries
$S_{jk}:= \sqrt{2 h}\sin( jk\pi h)$. In this case, $u_h(t)$ can be
computed exactly without the use of the sinc quadrature.
In fact, the matrix
representation of $L_h$ is given by 
$\widetilde L_h = S^{-1} \Lambda S$ with $\Lambda$ denoting the diagonal matrix
with diagonal $\Lambda_{jj} = 6h^{-2}(1-\cos(j\pi h))/(2+\cos(j\pi h))$,
$j=1\ldots, M$.   The matrix $\widetilde L_h $ takes coefficients of a
function $w\in H_h$ to those of $L_hw$.
Let $\overline V$ be the vector in $\RR^M$ defined by
$$\overline V_j= (v,\phi_j), \quad j=1,\ldots,M.$$
Then, the matrix representing $e^{-tL_h^\beta} w$ is thus given by 
$S^{-1} e^{-t\Lambda^\beta} S $ so the vector of coefficients representing
$e^{-tL_h^\beta}\pi_h v$ is  given by
$$ S^{-1} \mcD(t) S \overline V$$
where $\mcD(t)$ is the diagonal matrix with diagonal entries
$$\mcD(t)_{ii}= \frac {3e^{-t \Lambda_{ii}^\beta}}{h(2+\cos(i\pi
  h))},\quad i=1,\ldots,M.$$ 
The action of $S $ on a vector can be efficiently 
computed using the Fast Fourier Transform in $O(M\ln \ M)$ operations 
and $S^{-1}=S$.

To compute the solution $u(t)$ at the finite element nodes, 
the exact solution $u$ is approximated using the first 50000 modes of its Fourier representation.
The number of modes is chosen large enough such that it does not influence the space discretization error. 
Let $I_h$ denote the finite element interpolant operator associated with
$H_h$. 
As 
$$
\| u(t) - I_h u(t) \| \leq C t^{-\frac{1/4+\epsilon}{\beta}} h^2,$$
we report 
$\| I_hu(t) - u_h(t) \|$ in Tables 5.1--5.4. 

Table~\ref{table:init_semi_2} reports the $L^2$ error $e_{h_i}:=\|I_{h_i}u(t)-u_{h_i}(t)\|$ and observed rate of convergence
$$
OROC_i:= \ln(e_i/e_{i+1}) / \ln (h_i/h_{i+1})
$$
for different $\beta$ at time $t=0.5$.
In all cases, we observed $\|u(t)-u_h(t)\| \sim h^2$ as predicted by Theorem~\ref{thm:semi_init}, see also \eqref{e:num_rate}.
\begin{table}[hbt!]
 \begin{center}
    \begin{tabular}{|c|c|c|c|c|c|c|}
     \hline
     $h$ & \multicolumn{2}{|c|}{$\beta=0.25$}& \multicolumn{2}{|c|}{$\beta=0.5$} & \multicolumn{2}{|c|}{$\beta=0.75$} \\
     \hline
     $1/8$ & $6.29\times 10^{-4}$ & & $4.96\times 10^{-4}$ & &$1.14\times 10^{-4}$ &    \\
     \hline
     $1/16$  & $1.59\times 10^{-4}$ & $1.98$ &$1.25\times 10^{-4}$ & $1.99$ & $2.92\times 10^{-5}$ & $1.97$  \\
     \hline
     $1/32$ & $3.98\times 10^{-5}$ & $2.00$ &$3.12\times 10^{-5}$ & $2.00$& $7.33\times 10^{-6}$ & $1.99$  \\
     \hline
     $1/64$ & $9.95\times 10^{-6}$ & $2.00$ &$7.81\times 10^{-6}$ & $2.00$&$1.83\times 10^{-6}$ &  $2.00$  \\
     \hline
     $1/128$ & $2.49\times 10^{-6}$ & $2.00$ &$1.95\times 10^{-6}$ & $2.00$& $4.59\times 10^{-7}$ &$2.00$ \\
     \hline
    \end{tabular}
 \end{center}
  \caption{$L^2$ errors and observed rate of convergence (OROC) for different values of $\beta$.  
  The observed error decay is in accordance with Theorem~\ref{thm:semi_init}.}
  \label{table:init_semi_2}
\end{table}

\paragraph{A Parabolic Equation with Non-Homogeneous $f$ and Zero Initial Data.}
We now consider the one dimensional problem above but with $v=0$ and 
$$
f(x,t):=f(x):=\left \{\bal 2x, & \qquad x<0.5,\\
     2-2x,& \qquad x\geq 0.5.
         \eal \right.
$$
This choice implies that $f \in L^{\infty}(0,T;\dot{H}^{\frac{3}{2}-\epsilon})$ for every  $\epsilon>0$ so that
Corollary~\ref{cor:nonho_semi} predicts a rate of convergence $\min(2\beta+3/2,2)$.
 
In Table~\ref{tab:aroc_2}, we report
the asymptotic observed convergence
rate, computed as in Table~\ref{table:init_semi_2} for $t=0.5$. This rate is
defined  
by  OROC$:= \ln(e_{h_9}/e_{h_10}) / \ln2$ for $\beta>1/4$
and OROC$:= \ln(e_{h_{12}}/e_{h_{13}}) / \ln2)$ for $\beta<1/4$ where
$h_i=1/2^i$. The
finer mesh sizes were used in the case of $\beta<1/4$ to get closer to
the asymptotic convergence order.

\begin{table}[hbt!]
  \begin{center}
    \begin{tabular}{c|c|c|c|c|c|c|c|c|c}
   & \multicolumn{2}{|c|}{$\beta<0.25$}& \multicolumn{7}{c}{$\beta>0.25$} \\    
\hline
   $\beta $& 0.1 & 0.2 & 0.3 & 0.4 & 0.5& 0.6 & 0.7 & 0.8 & 0.9 \\   
   \hline
   OROC & 1.73 & 1.88 & 1.95 & 1.99 & 2.00 & 2.00 & 2.00 & 2.00 & 1.00\\
   THM & 1.7 & 1.9 & 2.0 & 2.0 & 2.0 & 2.0 & 2.0 & 2.0 & 2.0
 \end{tabular}
 \end{center}
  \caption{Observed rate of convergence (OROC) for different values of $\beta$ together the rates predicted by 
Corollary~\ref{cor:nonho_semi} (THM).}
  \label{tab:aroc_2}
\end{table}

\paragraph{A two dimensional homogenous initial value problem.}
Consider the unit square $\Omega:=(0,1)^2$, $L:=-\Delta$  
and the checkerboard initial data
\begin{equation}\label{eq:checkerboard}
  v(x_1,x_2)=\left \{\bal 1 & \qquad \text{if }(x_1-0.5)(x_2-0.5)>0, \\
     0 & \qquad \text{otherwise.}
         \eal \right. 
\end{equation}
Since we have $v\in \dot{H}^{1/2-2\epsilon}(\Omega)$ for all $\epsilon>0$, Theorem~\ref{thm:semi_init} guarantees an error decay
$$
\| u(t) - u_h(t) \| \leq C t^{-(\frac 3 4+\epsilon)/\beta} h^2.
$$ 
The subdivisions $\mathcal T_h$ are successive uniform refinement of $\Omega$ into squares.
Hence, taking advantage of the tensor product structure of the problem,
we compute $u_h(t)$ but using the two dimensional sine transform 
and  approximate $u(t)$ using, in this case, 
500 Fourier modes in each direction.
Table~\ref{table:2Derror} reports the values of $\|u(t)-u_h(t)\|$
together with the OROC for $t=0.5$ for several values of  $h$.
The numerical results reflect the  the error bound provided in Theorem~\ref{thm:semi_init}.

\begin{table}[hbt!]
 \begin{center}
    \begin{tabular}{|c|c|c|c|c|c|c|}
     \hline
     $h$ & \multicolumn{2}{|c|}{$\beta=0.25$}& \multicolumn{2}{|c|}{$\beta=0.5$} & \multicolumn{2}{|c|}{$\beta=0.75$} \\
     \hline
     $1/4$ & $2.21\times 10^{-2}$ & & $8.85\times 10^{-3}$ & &$2.15\times 10^{-3}$ &    \\
     \hline
     $1/8$  & $5.74\times 10^{-3}$ & $1.95$ &$2.74\times 10^{-3}$ & $1.69$ & $6.90\times 10^{-4}$ & $1.64$  \\
     \hline
     $1/16$ & $1.42\times 10^{-3}$ & $2.02$ &$7.39\times 10^{-4}$ & $1.89$& $1.84\times 10^{-4}$ & $1.91$  \\
     \hline
     $1/32$ & $4.32\times 10^{-4}$ & $1.72$ &$1.89\times 10^{-4}$ & $1.96$&$4.66\times 10^{-5}$ &  $1.98$  \\
     \hline
     $1/64$ & $1.36\times 10^{-4}$ & $1.66$ &$4.75\times 10^{-5}$ & $1.99$& $1.17\times 10^{-5}$ &$1.99$ \\
     \hline
     $1/128$ & $3.56\times 10^{-5}$ & $1.93$ &$1.19\times 10^{-5}$ & $2.00$& $2.93\times 10^{-5}$ &$2.00$ \\
     \hline
    \end{tabular}
 \end{center}
  \caption{$L^2$ error at $t=0.5$ and observed rate of convergence for different values of $\beta$ in the two dimensional case with checkerboard initial condition. As predicted by Theorem~\ref{thm:semi_init}, the $L^2$ error decays like $h^2$. }
  \label{table:2Derror}
\end{table}

\subsection{Sinc Quadrature.}
We now focus on the quadrature error estimate given in
Theorem~\ref{quad_theorem} and study the two different relations between
$k$ and $N$ discussed in Examples~\ref{Example_exp} and
\ref{Example_exp_balance}.  To do this we introduce an approximation to 
$\|\cE(\cdot,t)\|_{L^\infty(10,\infty)}$ defined by the following
procedure.
\begin{enumerate}[(i)]
\item  We examine the value of $|\cE(\lambda,t)|$ for
  $\lambda_j=10\mu^j$ for $j=0,1,\cdots,\mathcal N$. Here $\mu>1$ and $\mathcal N$ is
  chosen sufficiently large so that $|\cE(\lambda,t)|$ is monotonically
  decreasing when $\lambda \geq \lambda_{\mathcal N}$ (for $t=.5$, we take 
$\mu = \frac 32$ and $\mathcal N=40$).
\item We set $k:=\mathop{\mathrm{argmax}}_{j=1,...,\mathcal N}(|\cE(\lambda_j,t)|)$ and approximate 
$$
\|\cE(\cdot,t)\|_{L^\infty(10,\infty)} \approx \max_{l=1,...,\mathcal N} |\cE(\rho_l,0.5)|, \qquad \textrm{where }\rho_l := \lambda_{k-1}+ \frac{\lambda_{k+1}-\lambda_{k-1}}{\mathcal N}l.
$$
By adjusting $\mu$, $\mathcal N$ and $l$, we can obtain
$\|\cE(\cdot,t)\|_{L^\infty(10,\infty)}$ to any desired accuracy.
\end{enumerate}
In Figures~\ref{fig:ce1} and \ref{fig:ce2}, we report 
values of $\|\cE(\cdot,t)\|_{L^\infty(10,\infty)}$ as a function of $N$
obtained by runing the above algorithm with $\mathcal N$ and $\mu$ adjusted so
that the results are accurate to the number of digits reported.  
When considering Example~\ref{Example_exp_balance}, we choose $d=\pi/8$.

For Figure~\ref{fig:ce1}, we take $t=0.5$.   The blue lines give the
results for Example~\ref{Example_exp} while the red lines give the 
results for Example~\ref{Example_exp_balance}.
Except for the case of $\beta=.25$, the
Example~\ref{Example_exp_balance} are somewhat better.

For Figure~\ref{fig:ce2}, we take $N=32$ and report the errors as a
function of $t$.  In all cases, Example~\ref{Example_exp_balance} shows
significant improvement over Example~\ref{Example_exp} for small 
$t$.    
For Example~\ref{Example_exp_balance}, we used $d=\pi/8$ so that
$k$ could be computed as a function of $ N$.

\begin{figure}[hbt!]
 \begin{center}
\includegraphics[scale=0.7]{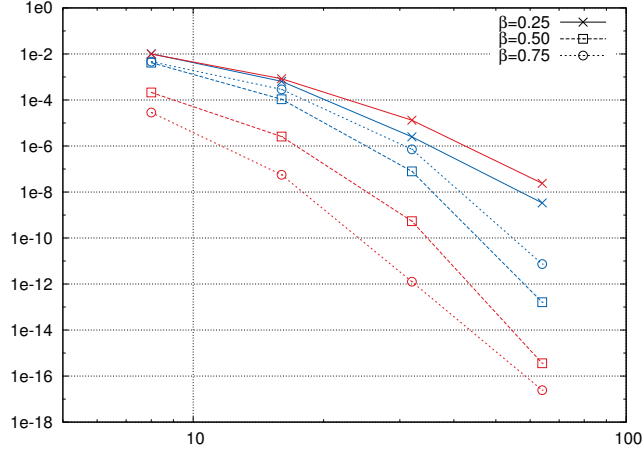}
 \end{center}
    \caption{$\|\cE(\cdot,1/2)\|_{L^\infty(10,\infty)}$ as a function of
      $N$ with Example~\ref{Example_exp} and
      Example~\ref{Example_exp_balance} reported in blue and red,
      respectively.}
    \label{fig:ce1}
    \end{figure}
 
\begin{figure}[hbt!]
 \begin{center}
\includegraphics[scale=0.7]{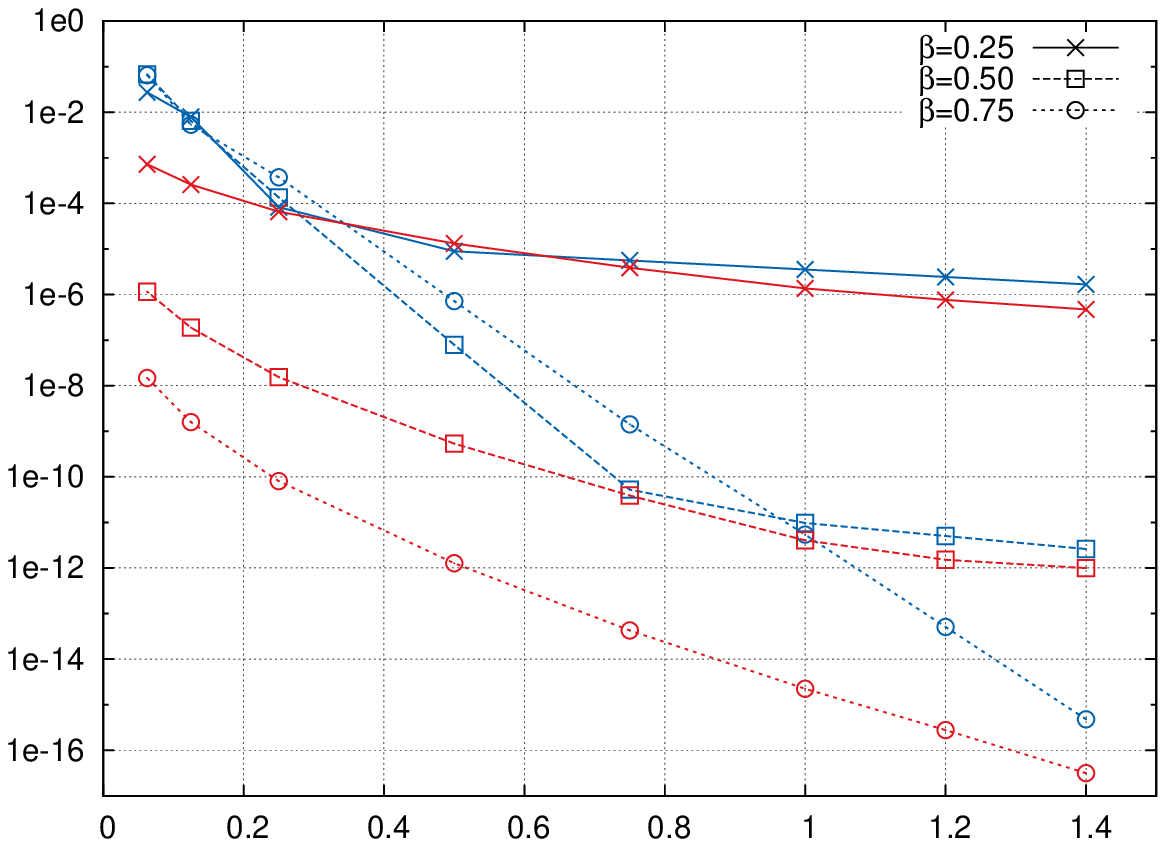}
 \end{center}
    \caption{$\|\cE(\cdot,t)\|_{L^\infty(10,\infty)}$ with $N=32$ 
as a function of $t$ with Example~\ref{Example_exp} and
      Example~\ref{Example_exp_balance} reported in blue and red,
      respectively.}
    \label{fig:ce2}
    \end{figure}

We consider again the two dimensional initial value problem  discussed above
but use the sinc quadrature approximation 
\eqref{eq:initq} with $N=40$ and $k=\ln(N)/(\beta N)$ (Example~\ref{Example_exp}). Here we use \emph{triangle} \cite{ISI:000175172500003}
to generate meshes such that each mesh is quasi-uniform and controlled by maximum area of cells. Approximation $Q_h^{N,k}(0.5)$ for different values of $\beta$ are provided in Figure~\ref{fig_png}, thereby illustrating the effect of $\beta$ on the diffusion strength. 
In addition, snapshots of $Q_h^{N,k}(t)$ at different times $t$ are provided in Figure~\ref{fig_evol_png}.

\begin{figure}[hbt!]
 \begin{center}
    \begin{tabular}{ccc}
\includegraphics[scale=.20]{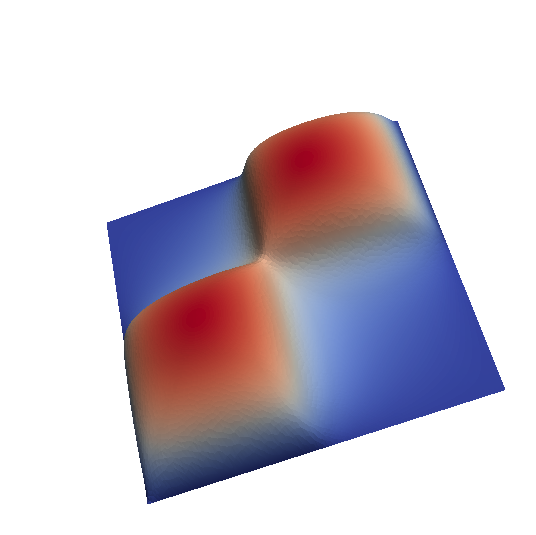} & \includegraphics[scale=.20]{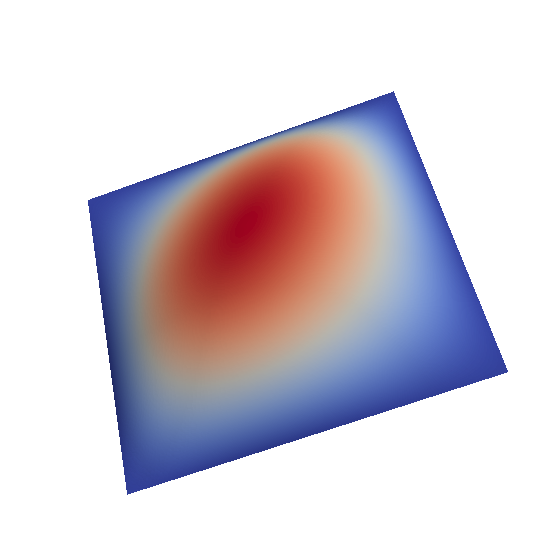}
&\includegraphics[scale=.20]{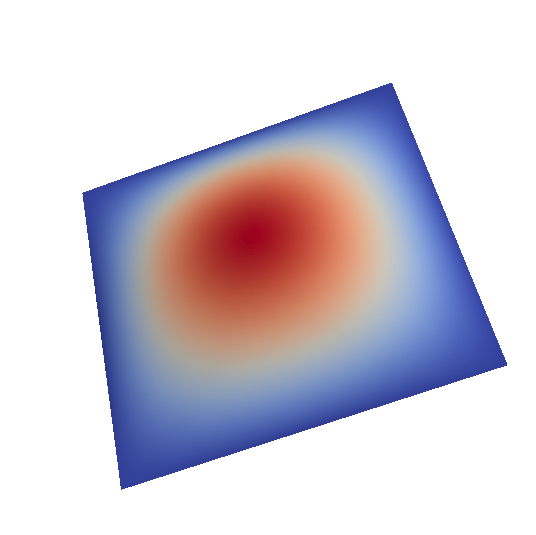}\\
$\beta=0.25$ & $\beta=0.5$ & $\beta=0.75$
    \end{tabular}
 \end{center}
    \caption{Approximations $Q_h^{N,k}(0.5)$ for initial data problem for different values of $\beta$. 
    The diffusion process is faster when increasing $\beta$.}
    \label{fig_png}
    \end{figure}

\begin{figure}[hbt!]
 \begin{center}
    \begin{tabular}{ccc}
    \includegraphics[scale=.20]{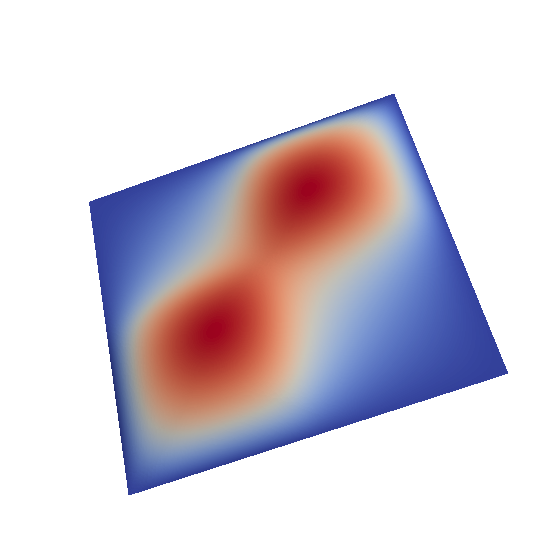} & \includegraphics[scale=.20]{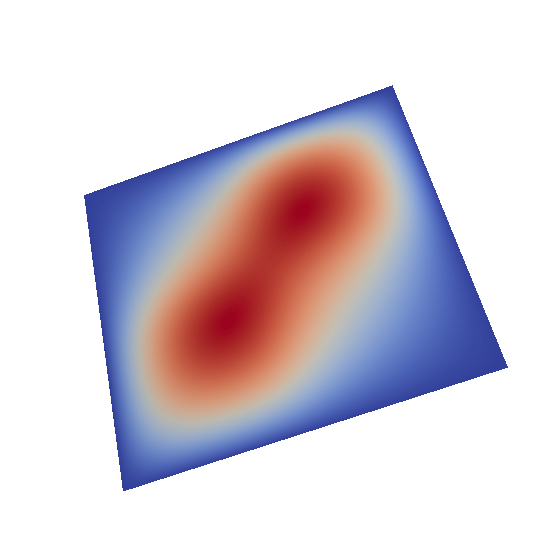}
&\includegraphics[scale=.20]{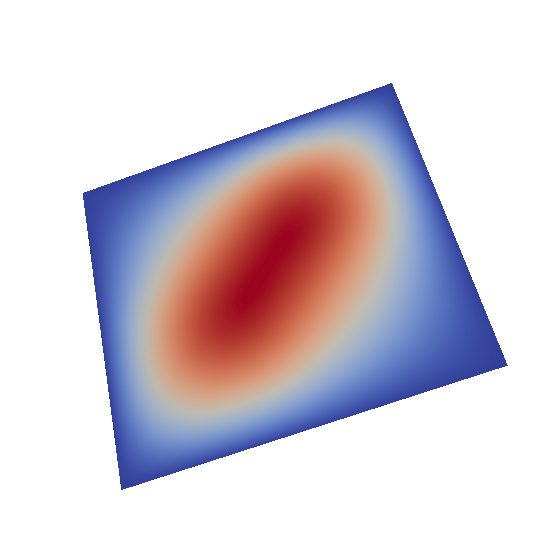}\\
$t=1.0$&$t=1.5$&$t=2.0$\\
    \end{tabular}
 \end{center}
    \caption{Evolution of the solution $Q_h^{N,k}(t)$ when $\beta=0.25$.}
    \label{fig_evol_png}
    \end{figure}

Finally, the total approximation errors, $\|Q_h^{N,k}(t)-u(t)\|$ at $t=0.5$, are reported in Table~\ref{table:total} for different values
of $\beta$.
The optimal order $2$ predicted by Corollary~\ref{c:total_error} is obtained for large $\beta$, while the asymptotic regime for $\beta=0.25$ was not reached in the computations.

\begin{table}[hbt!]
  \begin{center}
    \begin{tabular}{|c|c|c|c|c|c|c|}
     \hline
     $h^2$ & \multicolumn{2}{|c|}{$\beta=0.25$}& \multicolumn{2}{|c|}{$\beta=0.5$} & \multicolumn{2}{|c|}{$\beta=0.75$} \\
     \hline
     $0.02$ & $2.38\times 10^{-2}$ & & $1.47\times 10^{-3}$ & &$6.11\times 10^{-4}$ &    \\
     \hline
     $0.005$  & $6.20\times 10^{-3}$ & $1.94$ &$4.72\times 10^{-4}$ & $1.64$ & $1.66\times 10^{-4}$ & $1.88$  \\
     \hline
     $0.00125$ & $1.59\times 10^{-3}$ & $1.96$ &$1.21\times 10^{-4}$ & $1.96$& $4.32\times 10^{-5}$ & $1.94$  \\
     \hline
     $0.0003125$ & $4.26\times 10^{-4}$ & $1.90$ &$3.17\times 10^{-5}$ & $1.93$&$1.09\times 10^{-5}$ &  $1.99$  \\
     \hline
      $0.000078125$ & $1.13\times 10^{-4}$ & $1.91$ &$7.88\times 10^{-6}$ & $2.01$&$2.73\times 10^{-6}$ &  $2.00$  \\
     \hline
    \end{tabular}
 \end{center}
  \caption{Total approximation error $\|Q_h^{N,k}(t)-u(t)\|$ at $t=0.5$ and convergence rate for initial data \eqref{eq:checkerboard}
   with different values of $\beta$.
   The optimal order $2$ predicted by Corollary~\ref{c:total_error} is obtained for large $\beta$, while the asymptotic regime for $\beta=0.25$ was not reached in the computations.}
 \label{table:total}
\end{table}

\section*{Acknowledgment}
The first and second authors were
supported in part by
  the National Science Foundation through Grant DMS-1254618 while the 
second and third authors were supported in part by
  the National Science Foundation through Grant DMS-1216551.

\bibliographystyle{abbrv}

\end{document}